\documentclass[12pt]{amsart}

\usepackage{amsmath,amssymb,amsthm}
\usepackage[all]{xy}
\usepackage[dvipdfm,colorlinks=true]{hyperref}
\usepackage{epic,eepic}

\numberwithin{equation}{section}

\SelectTips{eu}{12}

\newtheorem{theorem}{Theorem}[section]
\newtheorem{proposition}[theorem]{Proposition}
\newtheorem{lemma}[theorem]{Lemma}
\newtheorem{corollary}[theorem]{Corollary}

\theoremstyle{definition}
\newtheorem{definition}[theorem]{Definition}

\newtheorem{question}[theorem]{Question}

\theoremstyle{remark}
\newtheorem{remark}[theorem]{Remark}

\newcommand{\Z}{\mathcal{Z}}
\newcommand{\RZ}{\mathbb{R}\mathcal{Z}}

\title{Golodness and polyhedral products for two dimensional simplicial complexes}

\author{Kouyemon Iriye}
\address{Department of Mathematics and Information Sciences, Osaka
Prefecture University, Sakai, 599-8531, Japan}
\email{kiriye@mi.s.osakafu-u.ac.jp}
\author{Daisuke Kishimoto}
\address{Department of Mathematics, Kyoto University, Kyoto, 606-8502, Japan}
\email{kishi@math.kyoto-u.ac.jp}
\thanks{K.I. is supported by JSPS KAKENHI (No. 26400094), and D.K. is supported by JSPS KAKENHI (No. 25400087)}

\subjclass[2010]{13F55,55P15}
\keywords{Stanley-Reisner ring, Golodness, neighborliness, moment-angle complex, polyhedral product, fat wedge filtration}

\begin{document}

\maketitle

\begin{abstract}
Golodness of 2-dimensional simplicial complexes is studied through polyhedral products, and combinatorial and topoogical characterization of Golodness of surface triangulations is given. An answer to the question of Berglund \cite{B} is also given so that there is a 2-dimensional simplicial complex which is rationally Golod but is not Golod over $\mathbb{Z}/p$.
\end{abstract}

\baselineskip 17pt


\section{Introduction}
Throughout this paper, let $K$ denote a simplicial complex on the vertex set $[m]:=\{1,\ldots,m\}$. The Stanley-Reisner ring of $K$ over a commutative ring $\Bbbk$ is by definition
$$\Bbbk[K]:=\Bbbk[v_1,\ldots,v_m]/(v_{i_1}\cdots v_{i_k}\,\vert\,\{i_1,\ldots,i_k\}\not\in K)$$
where we assume $|v_i|=2$ for convenience. The Stanley-Reisner rings have been a prominent object in combinatorial commutative algebra, and have been also a constant source of applications in a broad area of mathematics. In this paper we focus on the property of the Stanley-Reisner rings called Golodness which has been intensively studied in combinatorial commutative algebra. Originally Golodness of a graded $\Bbbk$-algebra $R$ is defined by a certain equality involving the Poincar\'e series of the cohomology of $R$, and Golod \cite{G} showed that this is equivalent to the triviality of products and (higher) Massey products in the positive part of the derived algebra $\mathrm{Tor}^+_{\Bbbk[v_1,\ldots,v_m]}(R,\Bbbk)$, where $R$ is the quotient of $\Bbbk[v_1,\ldots,v_m]$ and the ring structure is induced from that of the Koszul resolution of $\Bbbk$ over $\Bbbk[v_1,\ldots,v_m]$. Golodness of $K$ is defined by that of $\Bbbk[K]$, so we have:

\begin{definition}
A simplicial complex $K$ is Golod over $\Bbbk$ if all products and (higher) Massey products in $\mathrm{Tor}^+_{\Bbbk[v_1,\ldots,v_m]}(\Bbbk[K],\Bbbk)$ are trivial, where the ring structure is induced from the Koszul resolution of $\Bbbk$ over $\Bbbk[v_1,\ldots,v_m]$. 
\end{definition}

\begin{remark}
It is claimed in \cite{BJ} that the condition on (higher) Massey products in the definition of Golodness, i.e. $K$ is Golod over $\Bbbk$ if and only if all products in $\mathrm{Tor}^+_{\Bbbk[v_1,\ldots,v_m]}(\Bbbk[K],\Bbbk)$ are trivial. But this is recently shown to be false in \cite{K2}.
\end{remark}

We simply say that $K$ is Golod if it is Golod over any ring. Since the Stanley-Reisner ring of $K$ is isomorphic to the cohomology of a space which is now called the Davis-Januszkiewicz space \cite{DJ}, we can study the Stanley-Reisner rings by investigating the topology of spaces related to the Davis-Januszkiewicz spaces. In particular, there is a ring isomorphism
\begin{equation}
\label{Z-Tor}
H^*(\Z_K;\Bbbk)\cong\mathrm{Tor}^*_{\Bbbk[v_1,\ldots,v_m]}(\Bbbk[K],\Bbbk),
\end{equation}
so we aim at characterizing Golodness in terms of the topology of $\Z_K$, where $\Z_K$ is a space called the moment-angle complex for $K$ and is related to the Davis-Januszkiewicz space through a certain homotopy fibration \cite{DS}. The easiest case that $K$ is Golod is when $\Z_K$ is a co-H-space, and in \cite{IK} it is shown that the existence of a co-H-structure on $\Z_K$ is implied by triviality of a certain filtration of the real moment-angle complex $\RZ_K$ which is called the fat wedge filtration and will be explained later. As an application, we showed:

\begin{theorem}
[{\cite[Theorem 11.8]{IK}}]
\label{1-dim}
For a 1-dimensional $K$, the following conditions are equivalent; 
\begin{enumerate}
\item $K$ is Golod; 
\item the 1-skeleton of $K$ is chordal; 
\item The fat wedge filtration of $\RZ_K$ is trivial. 
\end{enumerate}
\end{theorem}

In this paper, we continue to study Golodness through the fat wedge filtration of $\RZ_K$ when $K$ is 2-dimensional, which is suggested in \cite{IK}. We will prove:

\begin{theorem}
\label{main1}
For a surface triangulation $K$, the following conditions are equivalent:
\begin{enumerate}
\item $K$ is Golod;
\item $K$ is 1-neighborly;
\item The fat wedge filtration of $\RZ_K$ is trivial.
\end{enumerate}
\end{theorem}

\begin{remark}
\begin{enumerate}
\item Although 1-neighborliness always implies Golodness of 2-dimensional simplicial complexes by Theorem \ref{neighborly} below, Golodness does not imply 1-neighborliness if we only assume that $K$ is 2-dimensional, that is, $K$ is a 2-dimensional simplicial complex which is not a surface triangulation. Indeed if $K$ is generated by facets $\{1,2,3\},\{3,4\}$ for $m=4$, then $\Z_K\simeq(S^1\times S^1)* S^1$, implying $K$ is Golod but is not 1-neighborly.
\item The implication (1) $\Rightarrow$ (2) of Theorem \ref{main1} is generalized when $K$ is a triangulation of a higher dimensional manifolds by Katth\"{a}n \cite{K2}.
\end{enumerate}
\end{remark}

We apply the method developed for a 2-dimensional $K$ in proving Theorem \ref{main1} to the question of Berglund \cite{B} which we recall here.

\begin{question}
\label{Berglund}
Is there a simplicial complex which is Golod over $\mathbb{Q}$ but is not Golod over $\mathbb{Z}/p$?
\end{question}

\begin{theorem}
\label{main2}
There is a triangulation of the 2-dimensional mod $p$ Moore space which is Golod over $\mathbb{Q}$ but is not Golod over $\mathbb{Z}/p$.
\end{theorem}

\begin{remark}
Katth\"{a}n \cite{K1} also studied Question \ref{Berglund} and gave an example of a simplicial complex $K$ which is Golod over $\mathbb{Q}$ and is not Golod over $\mathbb{Z}/p$. Our approach to Question \ref{Berglund} is purely topological and Katth\"{a}n's approach is purely algebraic.
\end{remark}


\section{Polyhedral products and their fat wedge filtrations}

This section collects facts on a polyhedral product $\Z_K(C\underline{X},\underline{X})$ and its fat wedge filtration that we will use. We refer to \cite{IK} for details. For a sequence of pairs of spaces $(\underline{X},\underline{A}):=\{(X_i,A_i)\}_{i\in[m]}$, the polyhedral product is defined by
$$\Z_K(\underline{X},\underline{A}):=\bigcup_{\sigma\in K}(\underline{X},\underline{A})^\sigma\quad(X_1\times\cdots\times X_m)$$
where $(\underline{X},\underline{A})^\sigma=Y_1\times\cdots\times Y_m$ for $Y_i=X_i$ ($i\in\sigma$) and $Y_i=A_i$ ($i\not\in\sigma$). Let $\underline{X}:=\{X_i\}_{i\in[m]}$ be a sequence of pointd spaces, and put $(C\underline{X},\underline{X}):=\{(CX_i,X_i)\}_{i\in[m]}$. We consider the special polyhedral product $\Z_K(C\underline{X},\underline{X})$, where it is the moment angle complex when $X_i=S^1$ for all $i$ and is the real moment-angle complex $\RZ_K$ when $X_i=S^0$ for all $i$. We introduce the fat wedge filtration of $\Z_K(C\underline{X},\underline{X})$. Put
$$\Z_K^i(C\underline{X},\underline{X}):=\{(x_1,\ldots,x_m)\in\Z_K(C\underline{X},\underline{X})\,\vert\,\text{at least }m-i\text{ of }x_i\text{'s are the basepoints}\}$$
for $i=1,\ldots,m$. We regard $\Z_{K_I}(C\underline{X}_I,\underline{X}_I)$ as a subspace of $\Z_K(C\underline{X},\underline{X})$ by using the basepoints of $X_i$'s, where $K_I:=\{\sigma\in K\,\vert\,\sigma\subset I\}$ is the full subcomplex of $K$ on $I$ and $\underline{X}_I=\{X_i\}_{i\in I}$ for $I\subset[m]$. Then we have
$$\Z_K^i(C\underline{X},\underline{X})=\bigcup_{I\subset[m],\,|I|=i}\Z_{K_I}(C\underline{X}_I,\underline{X}_I).$$
We now get a filtration
$$*=\Z_K^0(C\underline{X},\underline{X})\subset\Z_K^1(C\underline{X},\underline{X})\subset\cdots\subset\Z_K^m(C\underline{X},\underline{X})=\Z_K(C\underline{X},\underline{X})$$
which is called the fat wedge filtration of $\Z_K(C\underline{X},\underline{X})$. As in \cite{IK}, in studying the fat wedge filtration of $\Z_K(C\underline{X},\underline{X})$, that of $\RZ_K$ is especially important. We recall some properties of the fat wedge filtration of $\RZ_K$ shown in \cite{IK}. 

\begin{theorem}
[{\cite[Theorem 3.1]{IK}}]
\label{RZ-cone}
There is a map $\varphi_{K_I}\colon|K_I|\to\RZ_{K_I}^{|I|-1}$ for each $\emptyset\ne I\subset[m]$ such that $\RZ_K^i$ is $\RZ_K^{i-1}$ attached cones by the maps $\varphi_{K_I}\colon|K_I|\to\RZ_{K_I}^{i-1}\subset\RZ_K^{i-1}$ for all $I\subset[m]$ with $|I|=i$.
\end{theorem}

We say that the fat wedge filtration of $\RZ_K$ is trivial if $\varphi_{K_I}$ is null homotopic for any $\emptyset\ne I\subset[m]$.

\begin{theorem}
[{\cite[Theorem 1.2]{IK}}]
\label{RZ-decomp}
If the fat wedge filtration of $\RZ_K$ is trivial, then for any $\underline{X}$ there is a homotopy equivalence
$$\Z_K(C\underline{X},\underline{X})\simeq\bigvee_{\emptyset\ne I\subset[m]}|\Sigma K_I|\wedge\widehat{X}^I$$
where $\widehat{X}^I:=\bigwedge_{i\in I}X_i$.
\end{theorem}

We next recall two criteria for triviality of the fat wedge filtration of $\RZ_K$. A minimal non-face of $K$ is a subset $\sigma\subset[m]$ such that $\sigma\not\in K$ but $\sigma-v\in K$ for any $v\in\sigma$. In particular, if $\sigma$ is a minimal non-face of $K$, then $K\cup\sigma$ is a simplicial complex.

\begin{proposition}
[cf. {\cite[Theorem 7.2]{IK}}]
\label{RZ-factor}
Let $\widehat{K}$ be a simplicial complex obtained from $K$ by adding all minimal non-faces of $K$. The map $\varphi_K\colon|K|\to\RZ_K^{m-1}$ factors through the inclusion $|K|\to|\widehat{K}|$.
\end{proposition}

$K$ is called $k$-neighborly if any subsets with $k+1$ elements of $[m]$ is a simplex of $K$, or alternatively, the $k$-skeleton of $\Delta^{[m]}$ is included in $K$, where $\Delta^{[m]}$ denotes the full simplex on the vertex set $[m]$.

\begin{theorem}
[{\cite[Theorem 1.6]{IK}}]
\label{neighborly}
If $K$ is $\lceil\frac{\dim K}{2}\rceil$-neighborly, the fat wedge filtration of $\RZ_K$ is trivial.
\end{theorem}

In general we cannot describe the fat wedge filtration of $\Z_K(C\underline{X},\underline{X})$ by a cone decomposition as in Theorem \ref{RZ-cone}. But as for the moment-angle complex $\Z_K$ we have:

\begin{theorem}
[{\cite[Theorem 5.1]{IK}}]
\label{Z-cone}
There is a map $\widehat{\varphi}_{K_I}\colon|K_I|*S^{|I|-1}\to\Z_{K_I}^{|I|-1}$ for each $\emptyset\ne I\subset[m]$ such that $\Z_K^i$ is $\Z_K^{i-1}$ attached cones by the maps $\widehat{\varphi}_{K_I}\colon|K_I|*S^{i-1}\to\Z_{K_I}^{i-1}\subset\Z_K^{i-1}$ for all $I\subset[m]$ with $|I|=i$.
\end{theorem}

We say that the fat wedge filtration of $\Z_K$ is trivial if $\widehat{\varphi}_{K_I}$ is null homotopic for any $\emptyset\ne I\subset[m]$. Then in particular, if the fat wedge filtration of $\Z_K$ is trivial, $\Z_K$ is a co-H-space, implying $K$ is Golod. For a later use, we directly connect triviality of $\varphi_K$ to that of $\widehat{\varphi}_K$.

\begin{lemma}
\label{trivial-join}
If a cofibration $f\colon A\to X$ is null homotopic, so is the induced map
$$A*S^{m-1}=(CA\times S^{m-1})\cup(A\times D^m)\to((X\cup_fCA)\times S^{m-1})\cup(X\times D^m).$$
\end{lemma}

\begin{proof}
Since $f$ is a cofibration, the quotient map $X\cup_fCA\to X/A$ is a homotopy equivalence, so is also the map
$$((X\cup_fCA)\times S^{m-1})\cup(X\times D^m)\to(X/A\times S^{m-1})\cup(X\times D^m)$$ 
since these spaces are compatible homotopy puchouts of homotopy equivalent spaces, where the target is defined by the quotient map $q\colon X\to X/A$. By the null homotopy for $f$, we can shrink the composite
$$A*S^{m-1}\to((X\cup_fCA)\times S^{m-1})\cup(X\times D^m)\to(X/A\times S^{m-1})\cup(X\times D^m)$$
into $*\times D^m$, so it is null homotopic. Then since the second arrow is a homotopy equivalence, the proof is completed.
\end{proof}

\begin{proposition}
\label{Z-RZ}
$\widehat{\varphi}_K$ is null homotopic whenever so is $\varphi_K$.
\end{proposition}

\begin{proof}
By the construction in \cite{IK}, $\widehat{\varphi}_K$ factors through the map
$$|K|*S^{m-1}=(C|K|\times S^{m-1})\cup(|K|\times D^m)\to(\RZ_K\times S^{m-1})\cup(\RZ_K^{m-1}\times D^m),$$
where the second arrow is the restriction of the map $C|K|\times D^m\to\RZ_K\times D^m=(\RZ_K^{m-1}\cup_{\varphi_K}C|K|)\times D^m$. Thus the proof is done by Lemma \ref{trivial-join}.
\end{proof}


\section{Proof of Theorem \ref{main1} and \ref{main2}}

We first prove Theorem \ref{main1} by starting with a well-known description of Golodness.

\begin{proposition}
[Hochster \cite{H}]
\label{Hochster}
If $K$ is Golod over $\Bbbk$, then the inclusion $K_{I\cup J}\to K_I*K_J$ is trivial in cohomology with $\Bbbk$ coefficient for any $\emptyset\ne I,J\subset[m]$ with $I\cap J=\emptyset$.
\end{proposition}

We next consider the following easy lemma which immediately follows from Proposition \ref{RZ-factor}. Let $\mathrm{hodim}\,K$ and $\mathrm{conn}\,K$ denote the homotopy dimension and the connectivity of $K$, respectively.

\begin{lemma}
\label{null-homotopic}
If there is a simplicial complex $L$ satisfying $K\subset L\subset\widehat{K}$ and $\mathrm{hodim}\,K\le\mathrm{conn}\,L$, then $\varphi_K$ is null homotopic.
\end{lemma}

We specialize to the case $\mathrm{hodim}\,K\le 1$ which is essentially the same as the proof of the implication (2) $\Rightarrow$ (3) of Theorem \ref{main1} in \cite{IK}.

\begin{proposition}
\label{hodim<2}
If the 1-skeleton of $K$ is chordal and $\mathrm{hodim}\,K\le 1$, then $\varphi_K$ is null homotopic.
\end{proposition}

\begin{proof}
Let $L'$ be a simplicial complex obtained from $K$ by adding all 2-dimensional minimal non-faces. Then since the 1-skeleton of $K$ is chordal, every connected component of $L'$ is the 2-skeleton of the flag complex of a chordal graph, implying that every connected component of $L'$ is simply connected. Thus by connecting components of $L'$ by minimal non-faces of dimension 1, we obtain a simply connected simplicial complex $L$ satisfying $K\subset L\subset\widehat{K}$. Therefore the proof is completed by Lemma \ref{null-homotopic}.
\end{proof}

\begin{corollary}
\label{hodim<2-surface}
Suppose $K$ is a 2-dimensional simplicial complex such that $\mathrm{hodim}\,K_I\le 1$ for any $\emptyset\ne I\subsetneq[m]$ and the 1-skeleton of $K$ is chordal. Then the fat wedge filtration of $\RZ_K$ is trivial if and only if $\varphi_K$ is null homotopic.
\end{corollary}

\begin{proof}
The statement is equivalent to that under the same condition on $K$, $\varphi_{K_I}$ is null homotopic for any $\emptyset\ne I\subsetneq[m]$, which immediately follows from Proposition \ref{hodim<2}.
\end{proof}

\begin{proof}
[Proof of Theorem \ref{main1}]
The implication (2) $\Rightarrow$ (3) follows from Theorem \ref{neighborly}, and (3) $\Rightarrow$ (1) follows from Theorem \ref{RZ-decomp} since $\Z_K$ becomes a suspension. Then we show (1) implies (2). In the proof of \cite[Proposition 8.17]{IK} it is shown that if $K$ is Golod over some ring, then the 1-skeleton of $K$ is chordal. Since $K$ is a surface triangulation, we have $\mathrm{hodim}\,K_I\le 1$ for any $\emptyset\ne I\subsetneq[m]$. Then by Corollary \ref{hodim<2-surface} it remains to show that $\varphi_K$ is null homotopic. Assume (2) does not hold. Then there are vertices $v,w$ which are not connected by an edge. Since $K$ is a surface triangulation, the link of $K$, say $\mathrm{lk}_K(v)$, is homeomorphic to a circle. By considering the development of $K$ we see that $|\mathrm{lk}_K(v)|$ is a retract of the full subcomplex $|K_{[m]-\{v,w\}}|$. In particular, the composite 
$$|K|\to|(K_{[m]-\{v,w\}})*\{v,w\}|\cong|K_{[m]-\{v,w\}}|*|\{v,w\}|\to|\mathrm{lk}_K(v)|*|\{v,w\}|$$ 
is homotopic to the pinch map onto the top cell. Thus by Proposition \ref{Hochster}, $K$ is not Golod over $\mathbb{Z}/2$, completing the proof.
\end{proof}

We next prove Theorem \ref{main2}. 

\begin{proposition}
\label{Q-Golod}
Suppose $K$ is a 2-dimensional simplicial complex such that $\mathrm{hodim}\,K_I\le 1$ for any $\emptyset\ne I\subsetneq[m]$. If $|K|*S^{m-1}$ is rationally contractible, then $K$ is rationally Golod.
\end{proposition}

\begin{proof}
By Proposition \ref{Z-RZ} and Corollary \ref{hodim<2-surface}, $\widehat{\varphi}_{K_I}$ is null homotopic for any $\emptyset\ne I\subsetneq[m]$. If $|K|*S^{m-1}$ is rationally contractible, $(\widehat{\varphi}_K)_{(0)}$ is null homotopic. So since localization preserves homotopy cofibrations of simply connected spaces, we get that $(\Z_K)_{(0)}$ becomes a suspension by Theorem \ref{Z-cone}, competing the proof by \eqref{Z-Tor}.
\end{proof}

\begin{remark}
In the proof of Proposition \ref{Q-Golod} we must consider the fat wedge filtration of $\Z_K$ since that of $\RZ_K$ is a sequence of cofibrations of non-simply connected spaces which are not preserved by localization in general.
\end{remark}

We now construct a triangulation of the 2-dimensional mod $p$ Moore space which is rationally Golod but is not Golod over $\mathbb{Z}/p$. For $p=2$ let $M(p)$ be the triangulation of the projective plane illustrated in Figure 1 below. For $p>2$, let $M(p)$ be the triangulation of the 2-dimensional mod $p$ Moore space defined by: 
\begin{enumerate}
\item Take a $3p$-gon which is a $p$-fold cover of a triangle with vertices $v_1,v_2,v_3$;
\item Draw a $p$-gon with vertices $w_1,\ldots,w_p$, and give facets $\{v_1,v_2,w_i\},\{v_2,v_3,w_i\}$ for $i=1,\ldots,p$;
\item Drow one more $p$-gon with vertices $u_1,\ldots,u_p$, and give facets $\{v_3,v_1,u_i\},\{v_3,u_i,w_i\}$, $\{v_1,u_i,w_{i+1}\},\{u_i,w_i,w_{i+1}\}$ for $i\in\mathbb{Z}/p$;
\item Triangulate the $p$-gon $w_1\cdots w_p$ by giving facets $\{w_i,w_{i+1},w_p\}$ for $i=1,\ldots,p-2$.
\end{enumerate}
The case $p=4$ is illustrated in Figure 2 below. Let $G$ be a graph. For a vertex $v$, let $\mathtt{N}_G(v)$ be the neighborhood of $v$, that is, the set of vertices adjacent to $v$. We say that a set of vertices forms a clique in $G$ if each of two vertices in the set are adjacent. Let $\{a_1,\ldots,a_n\}$ be the vertex set of $G$. An ordering $a_1<\cdots<a_n$ is called a perfect elimination ordering if $\mathtt{N}_G(a_i)\cap\{a_{i+1},\ldots,a_n\}$ forms a clique in $G$ for $i=1,\ldots,n-1$. The following characterization of chordality is well known.

\begin{lemma}
\label{elimination}
A graph $G$ is chordal if and only if there is a perfect elimination ordering on the vertex set of $G$.
\end{lemma}

For $p>2$ we have
\begin{alignat*}{3}
&\mathtt{N}_{M(p)^{(1)}}(v_i)&&=\{v_{i+1},v_{i+2},w_1,\ldots,w_p,u_1,\ldots,u_p\}&\quad\text{for }i\in\mathbb{Z}/3,\\
&\mathtt{N}_{M(p)^{(1)}}(w_i)&&=\{v_1,v_2,v_3,w_{i-1},w_{i+1},u_{i-1},u_i\}\cup W_i&\quad\text{for }i\in\mathbb{Z}/p,\\
&\mathtt{N}_{M(p)^{(1)}}(u_i)&&=\{v_1,v_3,w_i,w_{i+1}\}&\quad\text{for }i\in\mathbb{Z}/p,
\end{alignat*}
where $W_1=W_{p-1}=\emptyset,W_p=\{w_2,\ldots,w_{p-2}\}$ and $W_i=\{w_p\}$ for $i=2,\ldots,p-2$. Then $u_1<\cdots<u_p<w_1<\cdots<w_p<v_1<v_2<v_3$ is a perfect elimination ordering, so the 1-skeleton of $M(p)$ is chordal by Lemma \ref{elimination}. On the other hand, one easily sees that any proper full subcomplex of $M(p)$ is of homotopy dimension $\le 1$ for $p>2$. For $p=2$, one can easily verify that the 1-skeleton of $M(2)$ is chordal, and any proper full subcomoplex of $M(2)$ is of homotopy dimension $\le 1$. We here record the properties of $M(p)$ that we have observed.

\begin{lemma}
\label{M}
The 1-skeleton of $M(p)$ is chordal for any $p$, and any proper full subcomplex of $M(p)$ is of homotopy dimension $\le 1$. 
\end{lemma}

We now prove the following which guarantees Theorem \ref{main2}.

\begin{theorem}
$M(p)$ is rationally Golod but is not Golod over $\mathbb{Z}/p$.
\end{theorem}

\begin{proof}
By Proposition \ref{Q-Golod} and Lemma \ref{M}, $M(p)$ is rationally Golod since $\Sigma M(p)$ is rationally contractible. We prove $M(p)$ is not Golod over $\mathbb{Z}/p$ by the same method as the proof of Theorem \ref{main1}. For $p=2$, $|\mathrm{lk}_{M(2)}(7)|$ is a retract of $|M(2)_{[5]}|$, and the composite
$$|M(2)|\to|M(2)_{[5]}*\{6,7\}|=|M(2)_{[5]}|*|\{6,7\}|\to|\mathrm{lk}_{M(2)}(7)|*|\{6,7\}|\cong S^2$$
is homotopic to the pinch map onto the top cell. Then the composite is non-trivial in the 2-dimensional mod 2 cohomology, hence $M(2)$ is not Golod over $\mathbb{Z}/2$. For $p>2$, we choose $w_1,w_{p-1}$ instead of $6,7$ for $M(2)$, and can show that $M(p)$ is Golod over $\mathbb{Z}/p$ by the same way. Therefore the proof is completed.
\end{proof}

\begin{figure}[htbp]
\begin{center}
\setlength\unitlength{1.1mm} 
\begin{tabular}{cc}
\begin{minipage}{0.35\hsize}
\begin{center}
\begin{picture}(50,50)(-25,-25)
\Thicklines
\drawline(21.65,12.5)(0,25)(-21.65,12.5)(-21.65,-12.5)(0,-25)(21.65,-12.5)(21.65,12.5)
\drawline(-21.65,12.5)(21.65,12.5)(0,-25)(-21.65,12.5)
\drawline(0,12.5)(-10.9,-6.4)(10.9,-6.4)(0,12.5)
\drawline(-21.65,-12.5)(0,0)
\drawline(21.65,-12.5)(0,0)
\drawline(0,25)(0,0)
\put(0,0){\circle*{1.5}}
\put(0,12.5){\circle*{1.5}}
\put(-10.9,-6.4){\circle*{1.5}}
\put(10.9,-6.4){\circle*{1.5}}
\put(21.65,12.5){\circle*{1.5}}
\put(0,25){\circle*{1.5}}
\put(-21.65,12.5){\circle*{1.5}}
\put(-21.65,-12.5){\circle*{1.5}}
\put(0,-25){\circle*{1.5}}
\put(21.65,-12.5){\circle*{1.5}}
\put(-3,14){$1$}
\put(-12,-11.2){$2$}
\put(10.4,-11.2){$3$}
\put(-3,26){$4$}
\put(-25,12.5){$5$}
\put(-25,-12.5){$6$}
\put(1.5,-28){$4$}
\put(23.2,-12.5){$5$}
\put(23.2,12.5){$6$}
\put(-1,-4.8){$7$}
\end{picture}
\caption{$M(2)$}
\end{center}
\end{minipage}
\hspace{8mm}
\begin{minipage}{0.45\hsize}
\begin{center}
\begin{picture}(60,60)(-30,-30)
\Thicklines
\drawline(30,0)(25.98,15)(15,25.98)(0,30)(-15,25.98)(-25.98,15)(-30,0)(-25.98,-15)(-15,-25.98)(0,-30)(15,-25.98)(25.98,-15)(30,0)
\drawline(30,0)(15,25.98)
\drawline(25.98,-15)(0,-30)
\drawline(-15,-25.98)(-30,0)
\drawline(-25.98,15)(0,30)
\drawline(-15,25.98)(-13,22.5)
\drawline(15,-25.98)(13,-22.5)
\drawline(-25.98,-15)(25.98,15)
\drawline(-13,22.5)(-22.5,-13)(13,-22.5)(22.5,13)(-13,22.5)
\drawline(-13,22.5)(-30,0)
\drawline(-22.5,-13)(-25.98,15)
\drawline(13,-22.5)(30,0)
\drawline(22.5,13)(25.98,-15)
\drawline(0,30)(22.5,13)
\drawline(15,25.98)(-13,22.5)
\drawline(0,-30)(-22.5,-13)
\drawline(-15,-25.98)(13,-22.5)
\put(-13,22.5){\circle*{1.5}}
\put(13,-22.5){\circle*{1.5}}
\put(-22.5,-13){\circle*{1.5}}
\put(22.5,13){\circle*{1.5}}
\put(-24.9,6.65){\circle*{1.5}}
\put(24.9,-6.65){\circle*{1.5}}
\put(-6.7,-24.9){\circle*{1.5}}
\put(6.7,24.9){\circle*{1.5}}
\put(30,0){\circle*{1.5}}
\put(25.98,15){\circle*{1.5}}
\put(15,25.98){\circle*{1.5}}
\put(0,30){\circle*{1.5}}
\put(-15,25.98){\circle*{1.5}}
\put(-25.98,15){\circle*{1.5}}
\put(-30,0){\circle*{1.5}}
\put(-25.98,-15){\circle*{1.5}}
\put(-15,-25.98){\circle*{1.5}}
\put(0,-30){\circle*{1.5}}
\put(15,-25.98){\circle*{1.5}}
\put(25.98,-15){\circle*{1.5}}
\put(31.5,-0.5){$v_1$}
\put(27,15.5){$v_2$}
\put(16,27){$v_3$}
\put(-1.2,32){$v_1$}
\put(-19.7,27){$v_2$}
\put(-30.7,16){$v_3$}
\put(-34.5,-0.5){$v_1$}
\put(-30,-18){$v_2$}
\put(-18,-29){$v_3$}
\put(-1.2,-33.5){$v_1$}
\put(16,-29){$v_2$}
\put(27,-18){$v_3$}
\put(-12.5,18.5){$w_1$}
\put(-18.8,-13.2){$w_2$}
\put(9,-19.8){$w_3$}
\put(14.5,12){$w_4$}
\put(-23.2,5){$u_1$}
\put(-7.5,-22.7){$u_2$}
\put(20,-6){$u_3$}
\put(5,21){$u_4$}
\end{picture}
\caption{$M(4)$}
\end{center}
\end{minipage}
\end{tabular}
\end{center}
\end{figure}

\end{document}